\documentclass[draft]{adm-paper}
\usepackage{amsfonts,amsmath,amsthm,amscd,amssymb,latexsym}

\begin{document}

\title{Representation of real numbers  by the alternating Cantor series}
\author{Symon Serbenyuk}
\shorttitle{Alternating Cantor series}
\shortauthor{S. Serbenyuk}
\address[S.~Serbenyuk]{Institute of Mathematics of the National Academy of Sciences of Ukraine, Tereschenkivska St.~3,~01601 Kyiv, Ukraine}{ simon.mathscience@imath.kiev.ua, simon6@ukr.net}{}%


\date{}

\theoremstyle{plain}
\newtheorem{theorem}{Theorem}
\newtheorem{lemma}{Lemma}
\newtheorem{proposition}{Proposition}
\newtheorem{corollary}{Corollary}
\newtheorem{definition}{Definition}
\theoremstyle{definition}
\newtheorem{example}{Example}
\newtheorem{remark}{Remark}

\maketitle

\begin{abstract}
The article is devoted to the alternating Cantor series.
 It is proved that any real number belonging to $[a_0-1;a_0]$, where $a_0=\sum^{\infty} _{k=1} {\frac{d_{2k}-1}{d_1d_2...d_{2k}}} $, has no more than two representations by the series and only numbers from countable subset of real numbers have two representations.
Geometry of the representation, properties of cylinders sets and semicylinders, simplest metric problems are investigated. Some applications of the series in fractal theory and relation between positive and alternating Cantor series are described. A shift  operator and its some applications, set of incomplete sums are studied. Necessary and sufficient conditions for representations of rational numbers by the   alternating Cantor series are formulated.
\end{abstract}

\subjclass{2001}{26A24,  11K55, 26A48, 11J83; }


\section*{Introduction}

At  present investigating of various numeral systems is useful  for development of metric, probability and fractal theories of real numbers, for study of fractal and others properties of mathematical objects with complicated local structure such as  continuous nowhere differentiable or singular functions,  random variables of Jessen-Wintner type, DP-transformations (transformations preserving the fractal Hausdorff-Besicovitch  \ \ dimension),  dynamical systems with chaotic trajectories, etc. \cite{Pra98, Pra92}.

There exist systems of representation of real numbers with finite or infinite alphabet, with redundant  digits or with zero redundancy.
 An s-adic numeral system is an example of real numbers encoding by finite alphabet and numbers representations by Ostrogradsky \cite{ABPT07, BPP09, P09}, Sylvester \cite{PZ09}, L\"uroth series \cite{PK09, PZh08},  or regular continued fractions, $A_2$-continued fractions \cite{DKP09}, polybasic $\tilde{Q}$-representation \cite{Pra98}, etc., are examples of  encoding of real numbers by infinite alphabet.  Representation \cite{C1869, B10,  R09, S13}  by the positive Cantor series  
\begin{equation}
\label{1}
\frac{\varepsilon_1}{d_1}+\frac{\varepsilon_2}{d_1d_2}+...+\frac{\varepsilon_n}{d_1d_2...d_n}+..., \varepsilon_n \in A_{d_n},
\end{equation}
 where $(d_n)$ is a fixed sequence of positive integers, $d_n > 1$, and $(A_{d_n} )$ is a sequence of sets  $A_{d_n} \equiv \{0,1,...,d_n-1\}$, is an example of polybasic numeral system with zero redundancy. The last-mentioned  encoding of real numbers has finite alphabet, when the sequence  $(d_n)$ is bounded. It is obvious that real numbers representation by the positive Cantor series is generalization of classic s-adic numeral system. It is easy to see that the series is "similar"  to  the following series 
$$
\sum^{\infty} _{n=1} {\frac{1}{a_1a_2...a_n}},
$$
where, $(a_n)$ is a monotone non-decreasing  sequence of positive integers, $a_1 \ge 1$. This series is Engel series \cite{PH06}. 

In 1869, Georg Cantor \cite{C1869} considered  expansions of real numbers in the series~\eqref{1}. Now there exist many papers  \cite{C1869, LT13,  B10,  R09, S13} devoted to study properties of representation of real numbers by the positive Cantor series but many problems related  to these series are not solved  completely. For example, criteria of representation of rational numbers, modeling of functions with  complicated local structure, etc.

Since expansions of numbers by the positive Cantor series are useful for study complicated objects of fractal analysis, I introduce a notion of the alternating Cantor series. The alternating Cantor series is not considered in publications earlier and generalizes nega-$s$-adic numeral system. In the paper I give a foundation of metric theory of representation of real numbers by the alternating Cantor series and consider some related problems of mathematical analysis.

\section{The object of research}

Let $(d_n)$ be a fixed sequence of numbers from $\mathbb N \setminus \{1\}$ and $(A_{d_n})$  be a sequense of sets $A_{d_n} \equiv \{0,1,2,...,d_n-~1\}$. 
\begin{definition} 
The sum
\begin{equation}
\label{nega-cs-def1}
-\frac{\varepsilon_1}{d_1}+\frac{\varepsilon_2}{d_1d_2}-\frac{\varepsilon_3}{d_1d_2d_3}+...+\frac{(-1)^n\varepsilon_n}{d_1d_2d_3...d_n}+...,
\end{equation}
where $\varepsilon_n \in A_{d_n}$ for any  $ n \in \mathbb N$, is called {\itshape the alternating Cantor series}. 
\end{definition}

A number $d_n$ is called називатимемо {\itshape the $n$th element} and  $\varepsilon_n $ is called {\itshape the $n$th digit} of the sum  \eqref{nega-cs-def1}.

Let us denote by $\Delta^{-D} _{\varepsilon_1\varepsilon_2...\varepsilon_n...}$ any number having representation \eqref{nega-cs-def1}. This notation is called {\itshape  representation of $x$  by the alternating Cantor series} or {\itshape   nega-D-representation}.

If sequence $(d_n)$ is a purely periodic sequence with simple period $(s)$, where $s>1$ is a fixed positive integer number, then the sum  \eqref{nega-cs-def1} will be transformed into the following sum
$$
-\frac{\varepsilon_1}{s}+\frac{\varepsilon_2}{s^2}-\frac{\varepsilon_3}{s^3}+...+\frac{(-1)^n\varepsilon_n}{s^n}+..., ~~~\varepsilon_n \in \{0,1,...,s-1\}.
$$
 The last-mentioned sum is nega-s-adic  representation of numbers from  $[-\frac{s}{s+1};\frac{1}{s+1}]$.

   Following sums are examples of alternating Cantor series.
\begin{enumerate}
\item 
$$
\sum^{\infty} _{n=1} {\frac{(-1)^n\varepsilon_n}{s^{\alpha_1+\alpha_2+...+\alpha_n}}}, 
$$
where $\alpha_n$  belongs to finite subset of positive integer numbers \ \ \   and $1<~s \in~\mathbb N$  is a fixed  number and  $\varepsilon_n \in \{0,1,...,s^{\alpha_n}-1\}$ for each $n \in \mathbb N$.
\item
$$
\sum^{\infty} _{n=1} {\frac{(-1)^n\varepsilon_n}{2 \cdot 3 \cdot ... \cdot (n+1)}}, ~~~\varepsilon_n \in \{0,1,...,n\}.
$$
\item
$$
\sum^{\infty} _{n=1} {\frac{(-1)^n\varepsilon_n}{p_1p_2...p_n}}, 
$$
where $(p_n)$ is an increasing sequence of all  prime numbers. 
\end{enumerate}

\begin{lemma}
Every alternating Cantor series is absolutely convergent series and its sum belongs to $[a_0-1;a_0]$, where 
$$
a_0=\sum^{\infty} _{n=1} {\frac{d_{2n}-1}{d_1d_2...d_{2n}}}=\sum^{\infty} _{n=1} {\frac{(-1)^{n+1}}{d_1d_2...d_n}}.
$$
\end{lemma}
\begin{proof} The statement of the lemma follows from the next  propositions
\begin{itemize} 
\item   the series
$$
\sum^{\infty} _{n=1} {\frac{\varepsilon_n}{d_1d_2...d_n}}
$$
is convergent;
\item 
$$
-1+\sum^{\infty} _{n=1} {\frac{(-1)^{n+1}}{d_1d_2...d_n}} \le S \le \sum^{\infty} _{n=1} {\frac{(-1)^{n+1}}{d_1d_2...d_n}},
$$
where $S$ is a sum of  \eqref{nega-cs-def1}. 
\end{itemize}
\end{proof}

\begin{lemma}
Following relations 
$$
 \frac{a_n-1}{d_1...d_n} \le r_n=\frac{(-1)^n}{d_1d_2...d_n}{\sum^{\infty} _{k=1} {\frac{(-1)^k \varepsilon_{n+k}}{d_{n+1}...d_{n+k}}}}\le \frac{a_n}{d_1...d_n}
$$
for $n=2m, m \in \mathbb N$ and
$$
  -\frac{a_n}{d_1d_2...d_n} \le r_n \le \frac{1-a_n}{d_1d_2...d_n} 
$$
 for $n=2m-1, m \in \mathbb N$,  where
$$
a_n=\sum^{\infty} _{k=1} {\frac{(-1)^{k+1}}{d_{n+1}d_{n+2}...d_{n+k}}},
$$
are true for residual  $r_n$ of the series \eqref{nega-cs-def1}.
\end{lemma}

\section{Representation of real numbers  by the alternating Cantor series}

\begin{lemma}
\label{nega-cs-lem1}
Each number $x \in [a_0-1;a_0]$ can be represented by the series~\eqref{nega-cs-def1}.
\end{lemma}
\begin{proof} It is obvious that 
$$
a_0=\Delta^{-D} _{0[d_2-1]0[d_4-1]0[d_6-1]0...},
$$
$$
a_0-1=\Delta^{-D} _{[d_1-1]0[d_3-1]0[d_5-1]0...}.
$$
Since $x$ be an arbitrary number from $(a_0-1;a_0)$, 
$$
-\frac{\varepsilon_1}{d_1}-\sum^{\infty} _{k=2} {\frac{d_{2k-1}-1}{d_1d_2...d_{2k-1}}}< x \le -\frac{\varepsilon_1}{d_1}+\sum^{\infty} _{k=1} {\frac{d_{2k}-1}{d_1d_2...d_{2k}}}
$$
with $0 \le \varepsilon_1 \le d_1-1$ and since
$$
[a_0-1;a_0]=I_0=\bigcup^{d_1-1} _{i=0} {\left[-\frac{i}{d_1}- \sum^{\infty} _{k=2} {\frac{d_{2k-1}-1}{d_1d_2...d_{2k-1}}}; -\frac{i}{d_1}+\sum^{\infty} _{k=1} {\frac{d_{2k}-1}{d_1d_2...d_{2k}}} \right]},
$$
it is obtained that
$$
- \sum^{\infty} _{k=2} {\frac{d_{2k-1}-1}{d_1d_2...d_{2k-1}}}<x+\frac{\varepsilon_1}{d_1}\le \sum^{\infty} _{k=1} {\frac{d_{2k}-1}{d_1d_2...d_{2k}}}.
$$
Let  $x+\frac{\varepsilon_1}{d_1}=x_1$, then following cases are  obtained:
\begin{enumerate}
 \item If equality
$$
 x_1= \sum^{\infty} _{k=1} {\frac{d_{2k}-1}{d_1d_2...d_{2k}}}.
$$
 holds, then
$$
x=\Delta^{-D} _{\varepsilon_1[d_2-1]0[d_4-1]0...} ~\mbox{or} ~ x=\Delta^{-D} _{[\varepsilon_1-1]0[d_3-1]0[d_5-1]0...}.
$$
\item If the equality is false, then $x=-\frac{\varepsilon_1}{d_1}+x_1$, where
$$
\frac{\varepsilon_2}{d_1d_2}-\sum^{\infty} _{k=2} {\frac{d_{2k-1}-1}{d_1d_2...d_{2k-1}}}\le x_1<\frac{\varepsilon_2}{d_1d_2}+\sum^{\infty} _{k=2} {\frac{d_{2k}-1}{d_1d_2...d_{2k}}}.
$$
\end{enumerate}
In the same way,  let  $x_2=x_1-\frac{\varepsilon_2}{d_1d_2}$, then
\begin{enumerate}
\item if equality
$$
x_2=\sum^{\infty} _{k=2} {\frac{d_{2k-1}-1}{d_1d_2...d_{2k-1}}}.
$$
holds, then
$$
x=\Delta^{-D} _{\varepsilon_1\varepsilon_2[d_3-1]0[d_5-1]0...} ~\mbox{or} ~ x=\Delta^{-D} _{\varepsilon_1[\varepsilon_2-1]0[d_4-1]0[d_6-1]0...}.
$$
\item In a different case
$$
x=-\frac{\varepsilon_1}{d_1}+\frac{\varepsilon_2}{d_1d_2}+x_2, ~\mbox{where}
$$
$$
-\frac{\varepsilon_3}{d_1d_2d_3}-\sum^{\infty} _{k=3} {\frac{d_{2k-1}-1}{d_1d_2...d_{2k-1}}}<x_2 \le -\frac{\varepsilon_3}{d_1d_2d_3}+\sum^{\infty} _{k=2} {\frac{d_{2k}-1}{d_1d_2...d_{2k}}}, ~\mbox{etc.}
$$
\end{enumerate}
So, for positive integer $m$,
$$
-\sum_{k>\frac{m+2}{2}} {\frac{d_{2k-1}-1}{d_1d_2...d_{2k-1}}}<x_m-\frac{(-1)^{m+1}\varepsilon_{m+1}}{d_1d_2...d_{m+1}}<\sum_{k>\frac{m+1}{2}} {\frac{d_{2k}-1}{d_1d_2...d_{2k}}}.
$$
Moreover, the following cases are possible:
\begin{enumerate}
\item 
$$
x_{m+1}=\begin{cases}
\sum_{k>\frac{m+2}{2}} {\frac{d_{2k-1}-1}{d_1d_2...d_{2k-1}}},&\text{if $m$ is an  odd number;}\\
\\
\sum_{k>\frac{m+1}{2}} {\frac{d_{2k}-1}{d_1d_2...d_{2k}}},&\text{if $m$ is an even number.}
\end{cases}
$$
In this case
$$
x=\Delta^{-D} _{\varepsilon_1\varepsilon_2...\varepsilon_{m+1}[d_{m+2}-1]0[d_{m+4}-1]0...} 
$$
or
$$
 x=\Delta^{-D} _{\varepsilon_1...\varepsilon_{m}[\varepsilon_{m+1}-1]0[d_{m+3}-1]0[d_{m+5}-1]0...}.
$$
\item If there does not exist $m \in \mathbb N$ such that the last-mentioned system is true, then   
$$
x=-\frac{\varepsilon_1}{d_1}+x_1=...=-\frac{\varepsilon_1}{d_1}+\frac{\varepsilon_2}{d_1d_2}-\frac{\varepsilon_3}{d_1d_2d_3}+...+\frac{(-1)^n\varepsilon_n}{d_1d_2...d_n}+x_n=....
$$
\end{enumerate}
Whence,
$$
x=\sum^{\infty} _{n=1} {\frac{(-1)^n \varepsilon_n}{d_1d_2...d_n}}.
$$
\end{proof}

\begin{lemma}
\label{nega-cs-lem2}
Numbers 
$$
x=\Delta^{-D} _{\varepsilon_1\varepsilon_2...\varepsilon_{m-1}\varepsilon_m \varepsilon_{m+1}...} ~\mbox{and}~ x^{'}=\Delta^{-D} _{\varepsilon_1\varepsilon_2...\varepsilon_{m-1}\varepsilon^{'} _m\varepsilon^{'} _{m+1}...},
$$
 where  $\varepsilon_m \ne \varepsilon^{'} _m$, are equal    iff one of the following systems
$$
\left\{
\begin{array}{rcl}
\varepsilon_{m+2i-1}&=&d_{m+2i-1}-1,\\
\varepsilon_{m+2i}&= 0 &=\varepsilon^{'} _{m+2i-1},\\
\varepsilon^{'} _{m+2i}&=&d_{m+2i}-1,\\
\varepsilon^{'} _m & = &\varepsilon_m-1;\\
\end{array}
\right.
\mbox{or}
\left\{
\begin{array}{rcl}
\varepsilon_{m+2i}&=&d_{m+2i}-1,\\
\varepsilon_{m+2i-1}&= 0 &=\varepsilon^{'} _{m+2i},\\
\varepsilon^{'} _{m+2i-1}&=&d_{m+2i-1}-1,\\
\varepsilon^{'} _m -1& = &\varepsilon_m;\\
\end{array}
\right.
$$
is satisfied for all $i \in \mathbb N$.
\end{lemma}
\begin{proof}
{\itshape Necessity.} Let $\varepsilon_m= \varepsilon^{'} _m +1 $. Then
$$
0=x-x^{'}=\Delta^{-D} _{\varepsilon_1\varepsilon_2...\varepsilon_{m-1}\varepsilon_m \varepsilon_{m+1}...}-\Delta^{-D} _{\varepsilon_1\varepsilon_2...\varepsilon_{m-1}\varepsilon^{'} _m\varepsilon^{'} _{m+1}...}=\frac{(-1)^m}{d_1d_2...d_m}+
$$
$$
+\frac{(-1)^{m+1}(\varepsilon_{m+1}-\varepsilon^{'} _{m+1})}{d_1d_2...d_{m+1}}+...+\frac{\varepsilon_{m+i}-\varepsilon^{'} _{m+i}}{d_1d_2...d_{m+i}}(-1)^{m+i}+...=
$$
$$
=\frac{(-1)^m}{d_1d_2...d_m}\left(1+\sum^{\infty} _{i=1} {\frac{(-1)^i (\varepsilon_{m+i}-\varepsilon^{'} _{m+i})}{d_{m+1}d_{m+2}...d_{m+i}}}\right).
$$
$$
v \equiv \sum^{\infty} _{i=1} {\frac{(-1)^i(\varepsilon_{m+i}-\varepsilon^{'} _{m+i})}{d_{m+1}d_{m+2}...d_{m+i}}} \ge - \sum^{\infty} _{i=1}{\frac{d_{m+i}-1}{d_{m+1}d_{m+2}...d_{m+i}}}=-1.
$$
The last-mentioned inequality is an equality in the following case 
$$
\varepsilon_{m+2i}=\varepsilon^{'} _{m+2i-1}=0~\mbox{ and} ~\varepsilon_{m+2i-1}=d_{m+2i-1}-1, ~\varepsilon^{'} _{m+2i}=d_{m+2i}-1.
$$

That is   conditions of the first system follow from $x=x^{'}$ in the case. It is easy to see that  second system  conditions follow from   $x=x^{'}$ on the assumption of $\varepsilon^{'} _m=\varepsilon_m+1$.  

It is obvious that {\itshape sufficiency} is true. 
\end{proof}

\begin{definition} The nega-D-representation  $\Delta^{-D} _{\varepsilon_1\varepsilon_2...\varepsilon_n...}$ of  the  number $x$ from $[a_0-1;a_0]$ is called  {\itshape periodic} if there exist numbers  $m \in \mathbb Z_0$ and $t \in~\mathbb N$ such that the  equality  $\varepsilon_{m+nt+j}=\varepsilon_{m+j}$ is true for arbitrary $n \in~\mathbb N$, $j\in \mathbb N$.The number is  denoted  by 
$$
\Delta^{-D} _{\varepsilon_1\varepsilon_2...\varepsilon_m(\varepsilon_{m+1}\varepsilon_{m+2}...\varepsilon_{m+t})}.
$$
The tuple of digits $(\varepsilon_{m+1}\varepsilon_{m+2}...\varepsilon_{m+t})$ of the last-mentioned  representation is called {\itshape period} and the number $t$ is {\itshape a length of the period}. The representation is  {\itshape a purely periodic}, if  $m=0$, and the representation is  {\itshape a mixed periodic}, if $m>0$.
\end{definition}

\begin{definition} The nega-D-representation  $\Delta^{-D} _{\varepsilon_1\varepsilon_2...\varepsilon_n...}$ of $x$ is  {\itshape a quasiperiodic}, if there exist numbers $m \in \mathbb Z_0$, $t \in \mathbb N$ and  functions $\phi_1, \phi_2,..., \phi_t$ (mapping sets $A_{d_n}$ in $A_{d_n}$ for each $n \in N$  ) such that
$$
x=\Delta^{-D} _{\varepsilon_1\varepsilon_2...\varepsilon_m\phi_1(d_{m+1})\phi_2(d_{m+2})...\phi_t(d_{m+t})\phi_1(d_{m+t+1})\phi_2(d_{m+t+2})...\phi_t(d_{m+2t})...}.
$$
\end{definition}

The following numbers are quasiperiodic:
$$
x_1=\Delta^{-D} _{\varepsilon_1\varepsilon_2...\varepsilon_n0[d_{n+2}-1][d_{n+3}-1]...[d_{n+i}-1]...}, 
$$
$$
 x_2=\Delta^{-D} _{\varepsilon_1\varepsilon_2...\varepsilon_n0[d_{n+1}-1]0[d_{n+3}-1]0[d_{n+5}-1]...}, ~\mbox{etc.}
$$
\begin{definition} A number $x \in I_0=[a_0-1;a_0]$ is called   {\itshape nega-D-rational number}, if it can be represented by 
$$
\Delta^{-D} _{\varepsilon_1\varepsilon_2...\varepsilon_{n-1}\varepsilon_n[d_{n+1}-1]0[d_{n+3}-1]0[d_{n+5}-1]...}
$$
or
$$
\Delta^{-D} _{\varepsilon_1\varepsilon_2...\varepsilon_{n-1}[\varepsilon_n-1]0[d_{n+2}-1]0[d_{n+4}-1]0[d_{n+6}-1...}.
$$
The rest of the numbers from $I_0$ are  {\itshape nega-D-irrational numbers}.
\end{definition}

The next proposition follows from Lemma   \ref{nega-cs-lem1} and Lemma  \ref{nega-cs-lem2}. 

\begin{theorem}
Every nega-D-irrational number has unique representation by the alternating Cantor series. Every nega-D-rational number has two  representations by the series \eqref{nega-cs-def1} such that for one from the representations  conditions $\varepsilon_{m+2i-1}=d_{m+2i-1}-1$ and $\varepsilon_{m+2i}=0$ are true and for the other representation for fixed   $m \in \mathbb Z_0$ and for any  $i \in \mathbb N$  conditions   $\varepsilon_{m+2i-1}=0$, $\varepsilon_{m+2i}=d_{m+2i}-1$ are true.
\end{theorem}

\begin{remark} Using this theorem, one can  consider  the nega-D-representation of a number as function of this number.  The function is well defined for nega-D-irrational numbers. Suppose one of representations of nega-D-rational number is not used. For example, the representation with conditions $\varepsilon_{m+2i-1}=0$, $\varepsilon_{m+2i}=d_{m+2i}-1$, where  $m \in \mathbb N$ is a fixed number and  $i \in \mathbb N$ is an arbitrary number. Then this function is well defined for nega-D-rational numbers too.

\end{remark}

\begin{remark} There exist sequences $(d_n)$ such that irrational number is a nega-D-rational number in the alternating
Cantor series representation. For example,
\end{remark}
$$
x=\sum^n _{i=1} {\frac{(-1)^i\varepsilon_i}{d_1d_2...d_i}}+\frac{(-1)^n}{d_1d_2...d_n}\left(-1- \sum^{\infty} _{j=1} {\frac{(-1)^j}{2\cdot 3 \cdot ... \cdot (j+1)}}\right)=
$$
$$
=\sum^n _{i=1} {\frac{(-1)^i\varepsilon_i}{d_1d_2...d_i}}+\frac{(-1)^n}{d_1d_2...d_n}\left(-1+\frac{1}{e}\right);
$$
$$
x=\sum^n _{i=1} {\frac{(-1)^i\varepsilon_i}{d_1d_2...d_i}}+\frac{(-1)^n}{d_1d_2...d_n}\left(-1- \sum^{\infty} _{j=1} {\frac{(-1)^j}{2\cdot 4 \cdot ... \cdot 2j}}\right)=
$$
$$
=\sum^n _{i=1} {\frac{(-1)^i\varepsilon_i}{d_1d_2...d_i}}+\frac{(-1)^{n+1}}{d_1d_2...d_n}\cdot \frac{\sqrt e}{e},
$$
because
$$
\Delta^{-D} _{\varepsilon_1...\varepsilon_n[d_{n+1}-1]0[d_{n+3}-1]0...}\equiv g_n+\frac{(-1)^n}{d_1d_2...d_n}\left(-1-\sum^{\infty} _{j=1} {\frac{(-1)^j}{d_{n+1}...d_{n+j}}}\right),
$$
$$
\Delta^{-D} _{\varepsilon_1...\varepsilon_{n-1}[\varepsilon_n-1]0[d_{n+2}-1]0...}\equiv g_n+\frac{(-1)^{n+1}}{d_1d_2...d_n}\left(1+\sum^{\infty} _{j=1} {\frac{(-1)^j}{d_{n+1}...d_{n+j}}}\right), 
$$
where
$$
g_n=\sum^n _{i=1} {\frac{(-1)^i\varepsilon_i}{d_1d_2...d_i}}.
$$

To avoid some inconveniences in the sequel one can modify notation (2) of the alternating Cantor series to 
\begin{equation}
\label{nega-cs-def2}
\sum^{\infty} _{n=1} {\frac{1+\varepsilon_n}{d_1d_2...d_n}(-1)^{n+1}},
\end{equation}
where $\varepsilon_n \in A_{d_n}$,  such that 
$$
[-1+a_0;a_0] \to [0;1], ~\mbox{where} ~ a_0=-\Delta^{-D} _{(1)}.
$$

It is easy to see that
$$
\inf\left(\sum^{\infty} _{n=1} {\frac{(-1)^{n+1}}{d_1d_2...d_n}}+\sum^{\infty} _{n=1} {\frac{(-1)^{n+1}\varepsilon_n}{d_1d_2...d_n}}\right)=g^{'}-\sum^{\infty} _{i=1} {\frac{d_{2i}-1}{d_1d_2...d_{2i}}}=0,
$$
$$
\sup\left(\sum^{\infty} _{n=1} {\frac{(-1)^{n+1}}{d_1d_2...d_n}}+\sum^{\infty} _{n=1} {\frac{(-1)^{n+1}\varepsilon_n}{d_1d_2...d_n}}\right)=g^{'} +\sum^{\infty} _{i=1} {\frac{d_{2i-1}-1}{d_1d_2...d_{2i-1}}}=1,
$$
where
$$
g^{'}=\sum^{\infty} _{n=1} {\frac{(-1)^{n+1}}{d_1d_2...d_n}}.
$$
The fact of representation of $x \in [0;1]$ by \eqref{nega-cs-def2} is denoted by $\Delta^{-(d_n)} _{\varepsilon_1\varepsilon_2...\varepsilon_n...}$. The last-mentioned notation is called  \ \ {\itshape nega-$d_n$-representation of  \  \  \  the number    $x \in [0;1]$}. Number $d_n$ is  called {\itshape $n$th element } and $\varepsilon_n=\varepsilon_n(x)$ is  {\itshape $n$th digit} of the sum     \eqref{nega-cs-def2}.

\section{ Some properties of representation of real numbers by the alternating Cantor series}

Let $x=\Delta^{-D} _{\varepsilon_1(x)\varepsilon_2(x)...\varepsilon_n(x)...}$ and $y=\Delta^{-D} _{\varepsilon_1(y)\varepsilon_2(y)...\varepsilon_n(y)...}$.
\begin{proposition}
\label{nega-cs-pro1}
For any  numbers $x$ and $y$ from $[-1+a_0;a_0]$ the inequality $x<y$ is true iff there exists $m$ such that   
$$
\varepsilon_n(x)=\varepsilon_n(y) ~\mbox{for} ~ n<2m ~\mbox{and}~ \varepsilon_{2m}(x)<\varepsilon_{2m}(y)
$$
or
$$
\varepsilon_n(x)=\varepsilon_n(y) ~\mbox{for} ~ n<2m-1 ~\mbox{and}~ \varepsilon_{2m-1}(x)>\varepsilon_{2m-1}(y).
$$
\end{proposition}

\begin{proposition}
Let \  $x_1=\Delta^{-D} _{\varepsilon_1\varepsilon_2...\varepsilon_{2k-1}(0)}$, \  $x_2=\Delta^{-D} _{\varepsilon_1\varepsilon_2...\varepsilon_{2k}(0)}$, \  $x_3=\Delta^{-D} _{\varepsilon_1\varepsilon_2...\varepsilon_{2k+1}(0)}$ and $\varepsilon_i \ne 0$ for  all $i=\overline{1,2k+1}$. Then the following two-sided inequality is true: 
$$
x_1<x_3<x_2.
$$
\end{proposition}
\begin{proof} The proposition follows from the equality
$$
x_3=x_1+\frac{1}{d_1d_2...d_{2k}}\left(\varepsilon_{2k}-\frac{\varepsilon_{2k+1}}{d_{2k+1}}\right)=x_2-\frac{\varepsilon_{2k+1}}{d_1d_2...d_{2k+1}}.
$$
\end{proof}

\begin{proposition}
Let \ $z_1=\Delta^{-D} _{\varepsilon_1\varepsilon_2...\varepsilon_{2k}(0)}$, \  $z_2=\Delta^{-D} _{\varepsilon_1\varepsilon_2...\varepsilon_{2k+1}(0)}$, \   $z_3=\Delta^{-D} _{\varepsilon_1\varepsilon_2...\varepsilon_{2k+2}(0)}$ and $\varepsilon_i \ne 0$ for all $i=\overline{1,2k+2}$. Then 
$$
z_2<z_3<z_1.
$$
\end{proposition}
\begin{proof} It is obvious that the proposition is true because
$$
z_3=z_1-\frac{1}{d_1d_2...d_{2k+1}}\left(\varepsilon_{2k+1}-\frac{\varepsilon_{2k+2}}{d_{2k+2}}\right)=z_2+\frac{\varepsilon_{2k+2}}{d_1d_2...d_{2k+2}}.
$$
\end{proof}

\section{Relation between representations of real numbers  by the positive  Cantor series  and the alternating   Cantor series }

Let $(d_n)$ be a fixed sequence of positive integer numbers such that  $d_n \ge1$.  For any  $ x \in [0;1]$ there exists sequence $(\alpha_n)$: $\alpha_n \in A_{d_n}$ such that
$$
x=\sum^{\infty} _{n=1} {\frac{\alpha_n}{d_1d_2...d_n}}\equiv \Delta^{D} _{\alpha_1\alpha_2...\alpha_n...}.
$$

It is obvious that 
$$
x=\frac{\alpha_1d_2+\alpha_2}{d_1d_2}+\frac{\alpha_3d_4+\alpha_4}{d_1d_2d_3d_4}+...+\frac{\alpha_{2n-1}d_{2n}+\alpha_{2n}}{d_1d_2...d_{2n}}+...,
$$
but this representation is a representation of  $x$ by the positive Cantor series with sequence elements  $(d^{'} _n)$, where $d^{'} _n=d_{2n-1}d_{2n}$.  

In fact, $0 \le \alpha_{2n-1}d_{2n}+\alpha_{2n}\le d_{2n-1}d_{2n}-1$ and therefore
\begin{equation}
\label{zv1}
x=\sum^{\infty} _{n=1} {\frac{\beta_n}{p_1p_2...p_n}}\equiv \Delta^{D^{'} _1}  _{\beta_1\beta_2...\beta_n...},
\end{equation}
where $\beta_n=\alpha_{2n-1}d_{2n}+\alpha_{2n}$, $p_n=d_{2n-1}d_{2n}$ for any $n \in \mathbb N$.

Now let us consider representation \eqref{nega-cs-def1}. Using the same technique, it is obtained that

$$
x=\frac{\varepsilon_2-\varepsilon_1d_2}{d_1d_2}+\frac{\varepsilon_4-\varepsilon_3d_4}{d_1d_2d_3d_4}+...+\frac{\varepsilon_{2n}-\varepsilon_{2n-1}d_{2n}}{d_1d_2...d_{2n}}+....
$$

But $(\varepsilon_{2n}-\varepsilon_{2n-1}d_{2n})$ belongs to   $\{0,1,...,d_{2n-1}d_{2n}-~1\}$ for not all  values  of $\varepsilon_{2n-1}$ and $\varepsilon_{2n}$. Thus  the nega-$(d_n)$-representation of the number~$x$  can be used in this case. 

Indeed, for 
$$
x=\sum^{\infty} _{n=1} {\frac{1+\delta_n}{d_1d_2...d_n}(-1)^{n+1}}\equiv \Delta^{-(d_n)} _{\delta_1\delta_2...\delta_n...},
$$
where
$$
\sum^{\infty} _{n=1} {\frac{(-1)^{n+1}}{d_1d_2...d_n}} \equiv \Delta^{-D} _{0[d_2-1]0[d_4-1]0...},
$$
it is obtained  that
$$
x=\sum^{\infty} _{n=1} {\frac{d_{2n}-1}{d_1d_2...d_{2n}}}+\frac{\delta_1d_2-\delta_2}{d_1d_2}+\frac{\delta_3d_4-\delta_4}{d_1d_2d_3d_4}+...+\frac{\delta_{2n-1}d_{2n}-\delta_{2n}}{d_1d_2...d_{2n}}+....
$$
Thus,  the  number  $(\delta_{2n-1}d_{2n}-\delta_{2n}+~d_{2n}-~1)$ belongs to  $\{0,1,...,d_{2n-1}d_{2n}-~1\}$ for real numbers nega-$(d_n)$-representation always and
\begin{equation}
\label{zv2}
x=\sum^{\infty} _{n=1} {\frac{(\delta_{2n-1}+1)d_{2n}-\delta_{2n}-1}{d_1d_2...d_{2n}}}\equiv \Delta^{D^{'} _1} _{\gamma_1\gamma_2...\gamma_n...},  
\end{equation}
where $\gamma_n=(\delta_{2n-1}+1)d_{2n}-\delta_{2n}-1=\delta_{2n-1}d_{2n}+d_{2n}-1-\delta_{2n}$.

The next proposition follows from  \eqref{zv1} and \eqref{zv2}. 

\begin{lemma}
The following functions are identity transformations of  $[0;1]$:
\end{lemma}
$$
x=\Delta^{D} _{\varepsilon_1\varepsilon_2...\varepsilon_n...}\stackrel{f}{\rightarrow}\Delta^{-(d_n)} _{\varepsilon_1[d_2-1-\varepsilon_2]...\varepsilon_{2n-1}[d_{2n}-1-\varepsilon_{2n}]...}=f(x)=y,
$$
$$
x=\Delta^{-(d_n)} _{\varepsilon_1\varepsilon_2...\varepsilon_n...}\stackrel{g}{\rightarrow}\Delta^{D} _{\varepsilon_1[d_2-1-\varepsilon_2]...\varepsilon_{2n-1}[d_{2n}-1-\varepsilon_{2n}]...}=g(x)=y.
$$
Therefore, the  following functions are DP-functions (functions preserving the fractal Hausdorff-Besicovitch dimension) on  $[0;1]$: 
$$
x=\Delta^{D} _{\varepsilon_1\varepsilon_2...\varepsilon_n...}\stackrel{f}{\rightarrow}\Delta^{-(d_n)} _{[d_1-1-\varepsilon_1]\varepsilon_2...[d_{2n-1}-1-\varepsilon_{2n-1}]\varepsilon_{2n}...}=f(x)=y,
$$
$$
x=\Delta^{-(d_n)} _{\varepsilon_1\varepsilon_2...\varepsilon_n...}\stackrel{g}{\rightarrow}\Delta^{D} _{[d_1-1-\varepsilon_1]\varepsilon_2...[d_{2n-1}-1-\varepsilon_{2n-1}]\varepsilon_{2n}...}=g(x)=y.
$$

\begin{lemma}
The following equalities are true: 
\begin{enumerate}
\item $\Delta^{D} _{\varepsilon_1\varepsilon_2...\varepsilon_n...}-\Delta^{-D} _{\varepsilon_1\varepsilon_2...\varepsilon_n...}\equiv 2\Delta^{D} _{\varepsilon_10\varepsilon_30...}\equiv-2\Delta^{-D} _{\varepsilon_10\varepsilon_30...};$

\item $\Delta^{D} _{\varepsilon_1\varepsilon_2...\varepsilon_n...}+\Delta^{-D} _{\varepsilon_1\varepsilon_2...\varepsilon_n...}\equiv 2\Delta^{D} _{0\varepsilon_20\varepsilon_4...} \equiv 2\Delta^{-D} _{0\varepsilon_20\varepsilon_4...};$

\item $\Delta^{D} _{\varepsilon_1\varepsilon_2...\varepsilon_n...}-\Delta^{-(d_n)} _{\varepsilon_1\varepsilon_2...\varepsilon_n...} \equiv 2\Delta^{D^{'}} _{\varepsilon_2\varepsilon_4\varepsilon_6...}-\Delta^{D^{'}} _{[d_2-1][d_4-1]...}$,  $\varepsilon_n \in A_{d_n}$;

\item $
\Delta^{D^{'}} _{\gamma_1\gamma_2...\gamma_n...} = \Delta^{D^{'}} _{\beta_1\beta_2...\beta_n...}+\Delta^{D^{'}} _{[d_2-1][d_4-1]...[d_{2n}-1]...}-2\Delta^{D^{'}} _{\varepsilon_2\varepsilon_4\varepsilon_6...}.$

\end{enumerate}
\end{lemma}

\section{Shift  operator}

Let $\mathcal F^{-D} _{[-1+a_0;a_0]}$ be a set of all nega-D-representations of real numbers  from  $[-1+~a_0;a_0]$.

  {\itshape A shift  operator $\hat \varphi$ of the sum \eqref{nega-cs-def1}} on the set $\mathcal F^{-D} _{[-1+a_0;a_0]}$ is defined by 
$$
\hat \varphi \left(\sum^{\infty} _{n=1} {\frac{(-1)^n \varepsilon_n}{d_1d_2...d_n}}\right)= \sum^{\infty} _{n=2} {\frac{(-1)^{n-1}\varepsilon_{n}}{d_2d_3...d_{n}}},
$$
$$
\hat \varphi(\Delta^{-D} _{\varepsilon_1\varepsilon_2...\varepsilon_n...})=\Delta^{-D_1} _{\varepsilon_2\varepsilon_3...\varepsilon_n...}=-d_1\Delta^{-D} _{0\varepsilon_2...\varepsilon_n...}.
$$
The operator generates a function $\hat \varphi$, that  
$$
\hat \varphi: [-1+a_0;a_0]\to [-a_0d_1;1-a_0d_1].
$$

n-fold application of the shift operator $\hat \varphi$ leads to the operator  $\hat \varphi^n$, that
$$
\hat \varphi^n \left(\sum^{\infty} _{n=1} {\frac{(-1)^n \varepsilon_n}{d_1d_2...d_n}}\right)= \sum^{\infty} _{k=1} {\frac{(-1)^{k}\varepsilon_{n+k}}{d_{n+1}...d_{n+k}}},
$$
$$
\hat \varphi^n(\Delta^{-D} _{\varepsilon_1\varepsilon_2\varepsilon_3...\varepsilon_n...})=\Delta^{-D_n} _{\varepsilon_{n+1}\varepsilon_{n+2}...}=(-1)^nd_1d_2...d_n\Delta^{-D} _{\underbrace{0...0}_{\mbox{$n$ }}\varepsilon_{n+1}\varepsilon_{n+2}...}.
$$

A generalized shift operator  $\hat \varphi_m$  of the sum \eqref{nega-cs-def1}  is defined by the following way: 
$$
\hat \varphi_m \left(\sum^{\infty} _{n=1} {\frac{(-1)^n \varepsilon_n}{d_1d_2...d_n}}\right)=-\frac{\varepsilon_1}{d_1}+...+\frac{(-1)^{m-1}\varepsilon_{m-1}}{d_1d_2...d_{m-1}}+\frac{(-1)^m\varepsilon_{m+1}}{d_1d_2...d_{m-1}d_{m+1}}+...,
$$
$$
\hat \varphi_m(\Delta^{-D} _{\varepsilon_1\varepsilon_2...\varepsilon_{m-1}\varepsilon_m\varepsilon_{m+1}...})=\Delta^{-D_{\bar m}} _{\varepsilon_1\varepsilon_2...\varepsilon_{m-1}\varepsilon_{m+1}...}. 
$$

\begin{remark}
An application of the shift operators $\hat \varphi$ or $\hat \varphi_m$ leads to transition to "other"  numeral system because for 
arbitrary \  number \ $x \in [a_0-1;a_0]$  sequences  $(\varepsilon_n)$  and  $(d_n)$ are fixed in sum \eqref{nega-cs-def1}. 
\end{remark}

It is easy to see that the operator $\hat \varphi$ has exactly  $d_1$ invariant points, that 
$$
-\frac{i}{d_1+1}, ~i=\overline{0,d_1-1}.
$$

The \ operator $\hat \varphi$ \ is \  not \ bijection because \ the points \ $\Delta^{-D} _{\varepsilon_1\varepsilon_2...\varepsilon_n...}$ $(\varepsilon_1=\overline{0,d_1-1})$ are preimages of the point $\Delta^{-D_1} _{\varepsilon_2\varepsilon_3...\varepsilon_n...}$.
 
If a sequence $(d_n)$ is  purely periodic with period of length $k$, then the mapping $\hat \varphi$ has periodic points of period of length $k, k \in \mathbb N$, i. e.
$$
 \hat \varphi^{k+j}(\Delta^{-D} _{(\varepsilon_1\varepsilon_2...\varepsilon_k)})=\hat \varphi^{kt+j}(\Delta^{-D} _{(\varepsilon_1\varepsilon_2...\varepsilon_k)}),~~~t=0,1,2,....          
$$

\begin{lemma} The following set   is invariant under the mapping $\hat \varphi$, if a sequence $(d_n)$ is a purely periodic sequence with simple period:
$$
C[-D,V]=\{x: x=\Delta^{-D} _{\varepsilon_1\varepsilon_2...\varepsilon_n...}, \varepsilon_n \in \{v_1,v_2,...,v_m\}\subset \{0,1,..., d_n-1\}\},
$$
where $v_1,v_2,...,v_m$ are fixed positive integer numbers, $1 < m\le d_n-1$, $d_n>2$. 

The set
$$
C[-D,V_n]=\{x: x=\Delta^{-D} _{\varepsilon_1\varepsilon_2...\varepsilon_n...}, \varepsilon_n \in V_n=\{v^{(n)} _1,v^{(n)} _2,...,v^{(n)} _m\}\subset A_n\}
$$
is an  invariant by the mapping $\hat \varphi^k$, if  sequences $(d_n)$ and  $(V_n)$ are purely periodic with  period of  length $k$. 
\end{lemma}

{\itshape A shift  operator  $\tilde \varphi$ of representation by the  alternating   Cantor series } of  $x=\Delta^{-D} _{\varepsilon_1\varepsilon_2...\varepsilon_n...}$ is defined by
$$
\tilde \varphi(\Delta^{-D} _{\varepsilon_1\varepsilon_2...\varepsilon_n...})=\Delta^{-D} _{\varepsilon_2...\varepsilon_n...}\equiv \sum^{\infty} _{n=1} {\frac{(-1)^n\varepsilon_{n+1}}{d_1d_2...d_n}}.
$$

It is obvious that the mapping  $\tilde \varphi$  is not always  well defined. In fact, for not all alternating   Cantor series the inequality $\varepsilon_{n+1}\le d_n-1$ is true. The next proposition follows from the last-mentioned inequality.

\begin{lemma}
\label{lm8}
The operator $\tilde \varphi$ is well defined iff for a sequence $(d_n) $ of elements of the sum \eqref{nega-cs-def1} for arbitrary $n \in \mathbb N$ the following inequality is true:
$$
d_{n+1}\le d_n.
$$
\end{lemma}

\begin{remark}
In Lemma \ref{lm8}, one can understand that operator is well defined in the wide sense, i. e., for arbitrary  $x \in [a_0-1;a_0]$. Really, for any sequence $(d_n) $ there exist points from $ [a_0-1;a_0]$ such that the function  $\tilde \varphi$ is well defined in these  points.
\end{remark}

\begin{lemma}
Let $x=\Delta^{-D} _{\varepsilon_1\varepsilon_2...\varepsilon_n...}$. If there exists  $m \in \mathbb N$, that  $\hat \varphi^m (x)=x$, then
$$
x=\left(1+\frac{1}{(-1)^md_1d_2...d_m-1}\right)\Delta^{-D} _{\varepsilon_1\varepsilon_2...\varepsilon_m(0)}.
$$
\end{lemma}
\begin{proof}
$$
x=\Delta^{-D} _{\varepsilon_1\varepsilon_2...\varepsilon_m\varepsilon_{m+1}...}=\Delta^{-D} _{\varepsilon_1\varepsilon_2...\varepsilon_m(0)}+\Delta^{-D} _{\underbrace{0...0}_{\mbox{$m$}}\varepsilon_{m+1}\varepsilon_{m+2}...}=
$$
$$
=\varphi^m (x)=x_m=(-1)^md_1d_2...d_m\Delta^{D}_{\underbrace{0...0}_{\mbox{$m$}}\varepsilon_{m+1}\varepsilon_{m+2}...}.
$$

The proposition of the lemma follows from the last-mentioned expression. 
\end{proof}

\begin{lemma}
Let $x$ be a fixed number. If there exist $m \in \mathbb Z_0$ and $c \in \mathbb N$ such that  $\hat \varphi^{m}(x)=\hat \varphi^{m+c}(x)$, then
$$
x=\frac{\Delta^{-D} _{\varepsilon_1\varepsilon_2...\varepsilon_m(0)}+(-1)^{c+1}d_{m+1}d_{m+2}...d_{m+c}\Delta^{-D} _{\varepsilon_1\varepsilon_2...\varepsilon_{m+c}(0)}}{1+(-1)^{c+1}d_{m+1}d_{m+2}...d_{m+c}}.
$$
\end{lemma}
\begin{proof}
The proposition follows from the next equality:
$$
x-\Delta^{-D} _{\varepsilon_1\varepsilon_2...\varepsilon_m(0)}=(-1)^cd_{m+1}d_{m+2}...d_{m+c}(x-\Delta^{-D} _{\varepsilon_1\varepsilon_2...\varepsilon_{m+c}(0)}).
$$
\end{proof}

\begin{lemma}
For arbitrary  $k \in \mathbb N$ the following equalities 
$$
\hat \varphi^{k}(x)=(-1)^kd_1d_2...d_kx+ (-1)^{k+1}d_1d_2...d_k\Delta^{-D} _{\varepsilon_1\varepsilon_2...\varepsilon_k(0)},
$$
$$
x=(-1)^k\frac{\hat \varphi^{k}(x)+(-1)^kd_1d_2...d_k\Delta^{-D} _{\varepsilon_1\varepsilon_2...\varepsilon_k(0)}}{d_1d_2...d_k},
$$
are true.
\end{lemma}

The next proposition follows from the last-mentioned lemma.

\begin{lemma}
For  arbitrary numbers $m \in \mathbb N$ and $c \in \mathbb N$ the equality 
$$
(-1)^cd_{m+1}d_{m+2}...d_{m+c}\cdot\hat \varphi^{m}(x)-\hat \varphi^{m+c}(x)=
$$
$$
=(-1)^{m+c}d_1d_2...d_{m+c}\cdot\Delta^{-D} _{\underbrace{0...0}_{\mbox{$m$}}\varepsilon_{m+1}\varepsilon_{m+2}...\varepsilon_{m+c}(0)}
$$
is true.
\end{lemma}

\begin{theorem}
The mapping  $\hat \varphi$ is decreasing on each first rank interval  $\nabla^{-D} _c=(\inf \Delta^{-D} _c;\sup \Delta^{-D} _c)$.
\end{theorem}
\begin{proof} Let points  $x_1=\Delta^{-D} _{c\varepsilon_2(x_1)\varepsilon_3(x_1)...\varepsilon_n(x_1)}$ and $x_2=\Delta^{-D} _{c\varepsilon_2(x_2)\varepsilon_3(x_2)...\varepsilon_n(x_2)}$ be an arbitrary points from the interval  $\nabla^{-D} _c$ such that  $x_1<x_2$.

$\hat \varphi(x_1)>\hat \varphi(x_2)$ because   $\varepsilon_n(\hat \varphi(x))=\varepsilon_{n+1}(x)$ and the Proposition \ref{nega-cs-pro1} is true for $x_1$ and $x_2$. 
\end{proof}

The next proposition follows from this theorem. 
\begin{corollary}
The mapping  $\hat \varphi$ has a derivative almost everywhere (with respect to
Lebesgue measure).
\end{corollary}

\begin{theorem}
The mapping $\hat \varphi$ is  continuous at each point of  first rank interval $\nabla^{-D} _c$ and endpoints of the  interval are points of discontinuity of the mapping.
\end{theorem}
\begin{proof}
Let $x=\Delta^{-D} _{c\varepsilon_2\varepsilon_3...\varepsilon_n}$ be an arbitrary nega-D-irrational point from $\nabla^{-D} _c$.
Let $(x_m)$ be an arbitrary sequence of points from $\nabla^{-D} _c$ such that $\lim_{m \to \infty} {x_m}=x$. Then
$$
\lim_{m \to \infty} {x_m}=x \Leftrightarrow \lim_{n \to \infty} {n_m}=\infty, ~\mbox{where}
$$
$n_m=\min\{n: \varepsilon_n(x_m)\ne \varepsilon_n(x)\}$.  The last-mentioned equivalence follows from definition and basic properties of the  nega-D-representation of real numbers. 

$\lim_{m \to \infty} {\hat \varphi(x_m)}=\hat \varphi(x)$ because  $\varepsilon_n(\hat \varphi(x))=\varepsilon_{n+1}(x)$. So, the mapping $\hat \varphi$ is  continuous in $x$.

Let $x_0$  be a certain nega-D-rational point from $\nabla^{-D} _c$, i. e.,  
$$
x_0=\Delta^{-D} _{c\varepsilon_2\varepsilon_3...\varepsilon_n[d_{n+1}-1]0[d_{n+3}-1]0...}=\Delta^{-D} _{c\varepsilon_2\varepsilon_3...\varepsilon_{n-1}[\varepsilon_n-1]0[d_{n+2}-1]0[d_{n+4}-1]0...}.
$$
At the same time
$$
\lim_{x \to x_0} {\hat \varphi(x)}=\Delta^{-D_1} _{\varepsilon_2...\varepsilon_n[d_{n+1}-1]0[d_{n+3}-1]0...}=\Delta^{-D_1} _{\varepsilon_2...\varepsilon_{n-1}[\varepsilon_n-1]0[d_{n+2}-1]0[d_{n+4}-1]0...}.
$$

Indeed, existence of left and  right  finite limits in each point of this interval follows from monotonicity and boundedness of mapping.

Let us consider a question of continuity  of $\hat \varphi$  in 
$$
x_1=\Delta^{-D} _{c[d_2-1]0[d_4-1]0...}=\Delta^{-D} _{[c-1]0[d_3-1]0[d_5-1]0...}=x_2, ~~~c \ne 0.
$$

The ends of the  interval $\nabla^{-D} _c$ are jump points of $\hat \varphi$ because 
$$
\hat \varphi(x_1)=\Delta^{-D_1} _{[d_2-1]0[d_4-1]0...}\ne \Delta^{-D_1} _{0[d_3-1]0[d_5-1]0...}=\hat \varphi(x_2). 
$$
\end{proof}

\begin{theorem}
If the mapping $\hat \varphi$  has derivative at point  $x=\Delta^{-D} _{\varepsilon_1\varepsilon_2...\varepsilon_n...}$, then  
$$
\left(\hat \varphi(x)\right)^{'}=-d_1.
$$
\end{theorem}
\begin{proof}
Let $\hat \varphi$ has derivative at  point  $x_0$, then for a  sequence  $(x_m)$  of numbers $\Delta^{-D}_{\varepsilon_1(x_0)\varepsilon_2(x_0)...\varepsilon_n(x_0)\varepsilon_{n+1}(x)...}$, where $\varepsilon_k(x) \ne \varepsilon_k(x_0)$ for all $k>n$,  it is  obtained that
$$
\left(\hat \varphi(x)\right)^{'}=\lim_{\Delta x \to 0} {\frac{\hat \varphi(x)-\hat \varphi(x_0)}{x-x_0}}=\lim_{\Delta x \to 0} {\frac{\Delta^{-D} _{\varepsilon_2(x)\varepsilon_3(x)...}-\Delta^{-D} _{\varepsilon_2(x_0)\varepsilon_3(x_0)...}}{\Delta x}}=
$$
$$
=-\lim_{\varepsilon_n(x) \to \varepsilon_n(x_0)} {\frac{\sum^{\infty} _{i=2} {\frac{(-1)^i\varepsilon_i(x)}{d_2...d_i}}-\sum^{\infty} _{i=2} {\frac{(-1)^i\varepsilon_i(x_0)}{d_2...d_i}}}{\sum^{\infty} _{j=1} {\frac{(-1)^j\varepsilon_j(x)}{d_1d_2...d_j}-\sum^{\infty} _{j=1} {\frac{(-1)^j\varepsilon_j(x_0)}{d_1d_2...d_j}}}}}=
$$
$$
=-\lim_{\varepsilon_n(x) \to \varepsilon_n(x_0)} {\frac{\sum^{\infty} _{j=1} {\frac{(-1)^j\varepsilon_j(x)}{d_1d_2...d_j}}-\sum^{\infty} _{j=1} {\frac{(-1)^j\varepsilon_j(x_0)}{d_1d_2...d_j}}+\frac{\varepsilon_1(x)-\varepsilon_1(x_0)}{d_1}}{\sum^{\infty} _{j=1} {\frac{(-1)^j\varepsilon_j(x)}{d_1d_2...d_j}-\sum^{\infty} _{j=1} {\frac{(-1)^j\varepsilon_j(x_0)}{d_1d_2...d_j}}}}}d_1=
$$
$$
=-d_1- d_1\lim_{n \to \infty} {\frac{0}{{\sum^{\infty} _{j=n+1} {\frac{(-1)^j\varepsilon_j(x)}{d_{n+1}d_2...d_j}}-\sum^{\infty} _{j=n+1} {\frac{(-1)^j\varepsilon_j(x_0)}{d_{n+1}d_2...d_j}}}}}=-d_1.
$$
\end{proof}
\begin{corollary}
The mapping $\hat \varphi^k$ has not derivative at any nega-D-rational point.  
\end{corollary}

\section{Representation of rational and irrational numbers }

An idea of the shift  operator of the sum \eqref{nega-cs-def1} is effective for formulation of criteria of rational numbers representation by the alternating series. Analogously to \cite{S13}, the main propositions can be formulated by the following way.

\begin{theorem}
A rational number $x=\frac{p}{q}$ from $[-1+a_0;a_0]$ has  finite expansion in the sum  \eqref{nega-cs-def1} iff there exists $n_0$, that $d_1d_2...d_{n_0} \equiv 0 (\mod q) $. 
\end{theorem}
\begin{corollary}
There exist sequences $(d_n)$, that every rational number has infinite representation by the series \eqref{nega-cs-def1}. For example, 
$$
\sum^{\infty} _{n=1} {\frac{(-1)^n\varepsilon_n}{2\cdot 3\cdot ... \cdot (n+1)}};~~~\sum^{\infty} _{n=1} {\frac{(-1)^n\varepsilon_n}{2\cdot 4\cdot ... \cdot 2n}}.
$$
\end{corollary}

\begin{lemma}
For each $n \in \mathbb N$ the equality
$$
\Delta^{-(d_n)} _{\varepsilon_1\varepsilon_2...\varepsilon_n...}+\Delta^{-D} _{\varepsilon_1\varepsilon_2...\varepsilon_n...}+\Delta^{-D} _{(1)}=0
$$
is true.
\end{lemma}
\begin{corollary}
There exist the alternating Cantor series \eqref{nega-cs-def2} such that a number  $\Delta^{-(d_n)} _{\varepsilon_1\varepsilon_2...\varepsilon_n(0)}$ is an irrational. 
\end{corollary}

The following propositions are equivalent.  

\begin{theorem}
A number $x_0 \in [0;1]$ is rational iff there exist $k \in \mathbb Z_0$ and $t \in \mathbb N$ such that 
$$
\hat \varphi^{k}(x)=\hat \varphi^{t}(x).
$$
\end{theorem}

\begin{theorem}
A number  $x_0 \in (-1+a_0;a_0)$ is rational iff for  its nega-D-representation $\Delta^{-D} _{\varepsilon_1\varepsilon_2...\varepsilon_n... }$ there exist $k \in \mathbb Z_0$ and $t \in \mathbb N$ $(k<t)$ such that
$$
\Delta^{-D} _{\underbrace{0...0}_{k}\varepsilon_{k+1}\varepsilon_{k+2}\varepsilon_{k+3}...}={(-1)^{t-k}d_{k+1}d_{k+2}...d_t}\Delta^{-D} _{\underbrace{0...0}_{t}\varepsilon_{t+1}\varepsilon_{t+2}\varepsilon_{t+3}...}.
$$
\end{theorem}

\begin{theorem}
A number  $x$ is rational iff the sequence  $(\hat \varphi^{k}(x))$, where $k=0,1,2,...$, has  at least  two identical terms 
of the sequence.
\end{theorem}

\section{Geometry and foundation of a metric theory of the nega-D-representation}

Let $(d_n)$ be a fixed sequence of positive integer numbers $d_n \ge 1$. 

Let $c_1, c_2,...,c_m$ be an ordered tuple of integer numbers such that $c_i~\in~\{0,1,...,d_i-~1\}$ for $i=~\overline{1,m}$.

\begin{definition} 
{\itshape Nega-D-cylinder of rank $m$ with the base $c_1c_2...c_m$} is a set $\Delta^{-D} _{c_1c_2...c_m}$ of all numbers from $[-1+a_0;a_0]$ such that  the first $m$ digits of the representation of the numbers are equal to $c_1,c_2,...,c_m$ respectively.
\end{definition}

\begin{lemma}
Nega-D-cylinder  is a  closed interval and 
$$
\Delta^{-D} _{c_1c_2...c_m}=\left[g_m+\frac{(-1)^m}{d_1d_2...d_m}(a_m-1); g_m+\frac{(-1)^m}{d_1d_2...d_m}a_m\right]
$$
for an even number $m$  and 
$$
\Delta^{-D} _{c_1c_2...c_m}=\left[g_m+\frac{(-1)^m}{d_1d_2...d_m}a_m; g_m+\frac{(-1)^m}{d_1d_2...d_m}(a_m-1)\right]
$$
for an odd number $m$, where
$$
a_m=\sup\sum^{\infty} _{j=1} {\frac{(-1)^j\varepsilon_{m+j}}{d_{m+1}d_{m+2}...d_{m+j}}}, ~g_m=\sum^{m} _{i=1} {\frac{(-1)^ic_i}{d_1d_2...d_i}}.
$$
\end{lemma}
\begin{proof}
Let  $m$ be an even number and  $x \in \Delta^{-D} _{c_1c_2...c_m}$, i. e. 
$$
x=\sum^{m} _{i=1} {\frac{(-1)^ic_i}{d_1d_2...d_i}}+\sum^{\infty} _{j=m+1} {\frac{(-1)^j\varepsilon_j}{d_1d_2...d_j}}, 
$$
where $ \varepsilon_j \in \{0,1,...,d_j-1\}$.
 Then
$$
x^{'}=g_m-\sum^{\infty} _{k=1} {\frac{d_{m+2k-1}-1}{d_1d_2...d_{m+2k-1}}}\le x\le g_m+\sum^{\infty} _{k=1} {\frac{d_{m+2k}-1}{d_1d_2...d_{m+2k}}}=x^{''}.
$$
Hence,  $x \in [x^{'};x^{''}]$ and $\Delta^{-D} _{c_1c_2...c_m}\subseteq [x^{'};x^{''}]$.

 $x \in \Delta^{-D} _{c_1c_2...c_m}$ and $x^{'} \in \Delta^{-D} _{c_1c_2...c_m} \ni x^{''}$ because
$$
\sum^{\infty} _{j=1} {\frac{d_{m+2j}-1}{d_1d_2...d_{m+2j}}}=\frac{1}{d_1d_2...d_m}\sup\sum^{\infty} _{j=1} {\frac{(-1)^j\varepsilon_{m+j}}{d_{m+1}d_{m+2}...d_{m+j}}},
$$
$$
-\sum^{\infty} _{j=1} {\frac{d_{m+2j-1}-1}{d_1d_2...d_{m+2j-1}}}=\frac{1}{d_1d_2...d_m}\inf\sum^{\infty} _{j=1} {\frac{(-1)^j\varepsilon_{m+j}}{d_{m+1}d_{m+2}...d_{m+j}}}.
$$
\end{proof}

\begin{lemma}
Cylinders $\Delta^{-D} _{c_1c_2...c_m}$ have the following properties.  
\begin{enumerate}
\item
$$
\inf \Delta^{-D} _{c_1c_2...c_m}=\begin{cases}
g_m-\frac{1}{d_1d_2...d_m}\sum^{\infty} _{j=1} {\frac{d_{m+2j-1}-1}{d_{m+1}d_{m+2}...d_{m+2j-1}}},&\text{ $m$ is an even,}\\
\\
g_m-\frac{1}{d_1d_2...d_m}\sum^{\infty} _{j=1} {\frac{d_{m+2j}-1}{d_{m+1}d_{m+2}...d_{m+2j}}},&\text{ $m$ is an odd.}
\end{cases}
$$
\item
$$
\sup \Delta^{-D} _{c_1c_2...c_m}=\begin{cases}
g_m+\frac{1}{d_1d_2...d_m}\sum^{\infty} _{j=1} {\frac{d_{m+2j}-1}{d_{m+1}d_{m+2}...d_{m+2j}}},&\text{ $m$ is an even;}\\
\\
g_m+\frac{1}{d_1d_2...d_m}\sum^{\infty} _{j=1} {\frac{d_{m+2j-1}-1}{d_{m+1}d_{m+2}...d_{m+2j-1}}},&\text{$m$ is an odd.}
\end{cases}
$$
\item
$$
| \Delta^{-D} _{c_1c_2...c_m}|=\frac{1}{d_1d_2...d_m}.
$$
\item
$$
 \Delta^{-D} _{c_1c_2...c_mc}\subset  \Delta^{-D} _{c_1c_2...c_m}.
$$
\item
$$
 \Delta^{-D} _{c_1c_2...c_m}=\bigcup^{d_{m+1}-1} _{c=0} { \Delta^{-D} _{c_1c_2...c_mc}}.
$$
\item
$$
\lim_{m \to \infty} { |\Delta^{-D} _{c_1c_2...c_m}|}=0.
$$
\item
$$
\frac{| \Delta^{-D} _{c_1c_2...c_mc_{m+1}}|}{| \Delta^{-D} _{c_1c_2...c_m}|}=\frac{1}{d_{m+1}}.
$$
\item
$$
\begin{cases}
\sup \Delta^{-D} _{c_1c_2...c_mc}=\inf  \Delta^{-D} _{c_1c_2...c_m[c+1]},&\text{if $m$ is an odd,}\\
\sup  \Delta^{-D} _{c_1c_2...c_m[c+1]}=\inf\Delta^{-D} _{c_1c_2...c_mc},&\text{if $m$ is an even, }
\end{cases}
$$
where $c \ne d_{m+1}-1$.
\item
$$
 \Delta^{-D} _{c_1c_2...c_m}\cap  \Delta^{-D} _{e_1e_2...e_m}=\begin{cases}
\Delta^{-D} _{c_1c_2...c_m},&\text{if $e_i=c_i ~ i=\overline{1,m}$;}\\
\varnothing,&\text{if $\exists i$, $i<m$, that $c_i \ne e_i$;}\\
\varnothing,&\text{if $\exists i$ that $c_i \ne e_i$, $c_m\ne e_m-1$,}
\end{cases}
$$
where $e_m\ne 0$  in the last case.
\item
$$
\bigcap^{\infty} _{m=1} {\Delta^{-D} _{c_1c_2...c_m}}=x \equiv \Delta^{-D} _{c_1c_2...c_m...}.
$$
\end{enumerate}
\end{lemma}
\begin{proof}
{\itshape Properties 1 and 2} follow immediately from \ the definition \ of \ $\Delta^{-D} _{c_1c_2...c_m}$. {\itshape The Property 3} is a corollary of these properties. {\itshape Properties  6 and 7} follow from the property 3. 

 {\itshape Property 4}. Let $m$  be an even positive integer number. It is proved that 
$$
\left\{
\begin{aligned}
\inf\Delta^{-D} _{c_1c_2...c_mc}&\ge\inf\Delta^{-D} _{c_1c_2...c_m},\\
\sup\Delta^{-D} _{c_1c_2...c_mc} & \le \sup\Delta^{-D} _{c_1c_2...c_m}.\\
\end{aligned}
\right.
$$
Really, $\inf\Delta^{-D} _{c_1c_2...c_mc}-\inf\Delta^{-D} _{c_1c_2...c_m}=g_m+$
$$
+\frac{(-1)^{m+1}c}{d_1d_2...d_{m+1}}+\frac{(-1)^{m+1}}{d_1d_2...d_{m+1}}\left(\frac{d_{m+3}-1}{d_{m+2}d_{m+3}}+\frac{d_{m+5}-1}{d_{m+2}...d_{m+5}}+...\right)-g_m-
$$
$$
-\frac{(-1)^m}{d_1d_2...d_m}\left(-\frac{d_{m+1}-1}{d_{m+1}}-\frac{d_{m+3}-1}{d_{m+1}d_{m+2}d_{m+3}}-\frac{d_{m+5}-1}{d_{m+1}...d_{m+5}}-...\right)=
$$
$$
=\frac{d_{m+1}-1-c}{d_1d_2...d_{m+1}}\ge 0, ~\mbox{where the inequality is an equality in }
$$
the  case of $c=d_{m+1}-1$.
$$
\sup\Delta^{-D} _{c_1...c_m}-\sup\Delta^{-D} _{c_1...c_mc}=\frac{(-1)^m}{d_1d_2...d_m}\left(\frac{d_{m+2}-1}{d_{m+1}d_{m+2}}+\frac{d_{m+4}-1}{d_{m+1}...d_{m+4}}+...\right)-
$$
$$
+g_m-g_m-\frac{(-1)^{m+1}c}{d_1...d_{m+1}}-\frac{(-1)^{m+1}}{d_1...d_{m+1}}\left(-\frac{d_{m+2}-1}{d_{m+2}}-\frac{d_{m+4}-1}{d_{m+2}...d_{m+4}}-...\right)=
$$
$$
=\frac{c}{d_1d_2...d_{m+1}}\ge 0, ~\mbox{ where the inequality is an equality in}
$$
 the case of $c=0$.

Similarly,  the last-mentioned inequality system is true  in the  case of an odd $m$.  

{\itshape The Property 5} follows from {\itshape the Property 4} and the definition of $\Delta^{-D} _{c_1c_2...c_m}$.

{\itshape Property 8}. In fact,  for an odd  $m$
$$
\sup \Delta^{-D} _{c_1c_2...c_mc}-\inf \Delta^{-D} _{c_1c_2...c_m[c+1]}=\frac{(-1)^{m+1}c}{d_1d_2...d_{m+1}}+\frac{(-1)^{m+1}}{d_1d_2...d_{m+1}}a_{m+1}-
$$
$$
-\frac{c+1}{d_1d_2...d_{m+1}}(-1)^{m+1}-\frac{(-1)^{m+1}}{d_1d_2...d_{m+1}}(a_{m+1}-1)=0.
$$

For an even $m$ 
$$
\sup \Delta^{-D} _{c_1c_2...c_m[c+1]}-\inf \Delta^{-D} _{c_1c_2...c_mc}=\frac{(-1)^{m+1}}{d_1d_2...d_{m+1}}(c+1)+
$$
$$
+\frac{(-1)^{m+1}}{d_1d_2...d_{m+1}}(a_{m+1}-1)-\frac{(-1)^{m+1}c}{d_1d_2...d_{m+1}}-\frac{(-1)^{m+1}}{d_1d_2...d_{m+1}}a_{m+1}=0.
$$

{\itshape The property 9} follows from  {\itshape properties 1, 2 and 8}.

{\itshape Property 10}. It is easy to see from {\itshape the property 4} that 
$$
\Delta^{-D} _{c_1} \subset \Delta^{-D} _{c_1c_2} \subset \Delta^{-D} _{c_1c_2c_3} \subset ... \Delta^{-D} _{c_1c_2...c_n} \subset ...,
$$
and from the last lemma and the Cantor's intersection theorem,  it follows that 
$$
\bigcap^{\infty} _{n=1} {\Delta^{-D} _{c_1c_2...c_n}}=x \equiv \Delta^{-D} _{c_1c_2...c_n...}.
$$ 
\end{proof}

\section{Simplest metric problems}

Let us  consider a set $\Delta^k _c$, that its elements have a fixed digit  $c \in A_{d_k}$ on  position $k$ of their nega-D-representations $\Delta^{-D} _{\varepsilon_1\varepsilon_2...\varepsilon_n...}$.

\begin{lemma}
The set  $\Delta^k _c$ $(k>1)$  is a union of cylinders of rank $k$. 
\end{lemma}
\begin{proof}
If $k=1$, then it is easy to see that $\Delta^k _c \equiv \Delta^{-D} _c$.

If $k=2$, then
$$
\Delta^2 _c =\Delta^{-D} _{0c} \cup \Delta^{-D} _{1c} \cup \Delta^{-D} _{2c} \cup...\cup \Delta^{-D} _{[d_1-1]c}.
$$

Let $k=n$, then
$$
\Delta^n _c =\Delta^{-D} _{\underbrace{00...00
}_{n-1}c} \cup \Delta^{-D} _{\underbrace{00...01}_{n-1}c} \cup...\cup \Delta^{-D} _{{[d_1-1][d_2-1]...[d_{n-1}-1]
}c}.
$$
\end{proof}

\begin{lemma}
The Lebesgue measure of  $\Delta^k _c$ is equal to $\frac{1}{d_k}$.
\end{lemma}
\begin{proof}
$$
\lambda(\Delta^k _c)=\sum^{d_1-1} _{c_1=0} ... \sum^{d_{k-1}-1} _{c_{k-1}=0} {|\Delta^{-D} _{c_1c_2...c_{k-1}c}|}=
$$
$$
=\frac{1}{d_k}\sum^{d_1-1} _{c_1=0} ... \sum^{d_{k-1}-1} _{c_{k-1}=0} {|\Delta^{-D} _{c_1c_2...c_{k-1}}|}=\frac{1}{d_k}.
$$
\end{proof}

\begin{corollary}
The Lebesgue measure of the set $\Delta^k _{\overline{c}}$ of numbers such that  its elements have a digit $\varepsilon_k \ne c$ on  position $k$ of their nega-D-representations $\Delta^{-D} _{\varepsilon_1\varepsilon_2...\varepsilon_n...}$ is equal to $1-\frac{1}{d_k}$.
\end{corollary}

\begin{lemma}
 Diameter of the  set $\Delta^k _c$ is calculated by the following formula 
$$
d(\Delta^k _c)=\frac{d_1d_2...d_k-d_k+1}{d_1d_2...d_k}.
$$
\end{lemma}
\begin{proof}
Let $k$ be an even number, 
$$
a_k=\sup\sum^{\infty} _{j=1} {\frac{(-1)^j\varepsilon_{k+j}}{d_{k+1}d_{k+2}...d_{k+j}}}. ~\mbox{Then} ~d(\Delta^k _c)=\max\sum^{k-1} _{i=1} {\frac{(-1)^i\varepsilon_i}{d_1d_2...d_i}}+
$$
$$
+\frac{(-1)^kc}{d_1d_2...d_k}+\frac{(-1)^k}{d_1d_2...d_k}a_k-\min\sum^{k-1} _{i=1} {\frac{(-1)^i\varepsilon_i}{d_1d_2...d_i}}-\frac{(-1)^kc}{d_1d_2...d_k}-
$$
$$
-\frac{(-1)^k}{d_1d_2...d_k}(a_k-1)=\left(\frac{d_2-1}{d_1d_2}+\frac{d_4-1}{d_1...d_4}+...+\frac{d_{k-2}-1}{d_1...d_{k-2}}\right)+
$$
$$
+\left(\frac{d_1-1}{d_1}+\frac{d_3-1}{d_1d_2d_3}+...+\frac{d_{k-1}-1}{d_1...d_{k-1}}\right)+\frac{(-1)^k}{d_1...d_k}=
$$
$$
=1-\frac{1}{d_1...d_{k-1}}+\frac{(-1)^k}{d_1d_2...d_k}=\frac{d_1d_2...d_k-d_k+1}{d_1...d_k}.
$$

Let  $k$ be an odd number, then 
$$
d(\Delta^k _c)=\left(\frac{d_2-1}{d_1d_2}+\frac{d_4-1}{d_1...d_4}+...+\frac{d_{k-1}-1}{d_1...d_{k-1}}\right)-
$$
$$
-\left(-\frac{d_1-1}{d_1}-\frac{d_3-1}{d_1d_2d_3}-...-\frac{d_{k-2}-1}{d_1...d_{k-2}}\right)+\frac{(-1)^{k+1}}{d_1d_2...d_k}=
$$
$$
=1-\frac{1}{d_1d_2...d_{k-1}}+\frac{1}{d_1d_2...d_k}=\frac{d_1d_2...d_k-d_k+1}{d_1...d_k}.
$$
\end{proof}

Let $(c_1,c_2,...,c_m)$ and $(k_1,k_2,...,k_m)$ are fixed tuples of positive integer numbers such that  $c_i \in A_{d_{k_i}}$, $i=\overline{1,m}$, $0<k_1<k_2<...<k_m$.

\begin{lemma}
 The Lebesgue measure of the set $\Delta^{k_1k_2...k_m} _{c_1c_2...c_m}$ of numbers such that  its elements have a digit $c_i$ on  position $k_i$ $(i=\overline{1,m})$ of their nega-D-representations $\Delta^{-D} _{\varepsilon_1\varepsilon_2...\varepsilon_n...}$  is calculated by formula 
$$
\lambda\left(\Delta^{k_1k_2...k_m} _{c_1c_2...c_m}\right)=\prod^m _{i=1} {\frac{1}{d_{k_i}}}.
$$
\end{lemma}
\begin{proof}  Let us  consider the set $\Delta^{k_1k_2} _{c_1c_2}$, that \ fixed \  digits \ $c_1 \in A_{d_{k_1}}$, \  $c_2 \in A_{d_{k_2}}$ are situated on positions  $k_1$ and $k_2$ respectively in  nega-D-representations of elements of the set.
$$
\lambda\left(\Delta^{k_1k_2} _{c_1c_2}\right)=\frac{1}{d_{k_2}}\cdot\frac{d_{k_2-1}}{d_{k_2-1}}\cdot...\cdot\frac{d_{k_1+1}}{d_{k_1+1}}|\Delta^{k_1} _{c_1}|=\frac{1}{d_{k_2}}\cdot\frac{1}{d_{k_1}}=\lambda\left(\Delta^{k_1} _{c_1}\right)\cdot\lambda\left(\Delta^{k_2} _{c_2}\right).
$$

$$
\lambda\left(\Delta^{k_1k_2...k_m} _{c_1c_2...c_m}\right)=\frac{1}{d_{k_m}}\left|\Delta^{k_1k_2...k_{m-1}} _{c_1c_2...c_{m-1}}\right|=\frac{1}{d_{k_m}}\cdot\frac{1}{d_{k_{m-1}}}\left|\Delta^{k_1k_2...k_{m-2}} _{c_1c_2...c_{m-2}}\right|=...=
$$
$$
=\frac{1}{d_{k_m}}\cdot \frac{1}{d_{k_{m-1}}}\cdot ...\cdot \frac{1}{d_{k_1}}=\frac{1}{d_{k_1}d_{k_2}...d_{k_m}}.
$$
\end{proof}

\begin{corollary}
Sets $\Delta^{k_1k_2...k_m} _{c_1c_2...c_m}$ are metrically independent, i. e. 
$$
\lambda\left(\Delta^{k_1k_2...k_m} _{c_1c_2...c_m}\right)=\lambda\left(\bigcap^m _{i=1} {\Delta^{k_i} _{c_i}}\right)=\prod^m _{i=1} {\lambda\left(\Delta^{k_i} _{c_i}\right)}.
$$
\end{corollary}

\begin{lemma} 
Diameter  $d\left(\Delta^{k_1k_2...k_m} _{c_1c_2...c_m}\right)$  of the set  $\Delta^{k_1k_2...k_m} _{c_1c_2...c_m}$ is  calculated by formula 
$$
d\left(\Delta^{k_1k_2...k_m} _{c_1c_2...c_m}\right)=1-\sum^{m} _{i=1} {\frac{d_{k_i}-1}{d_1d_2...d_{k_i}}}.
$$
\end{lemma}
\begin{proof} Let   $K=\{k_1,k_2,...,k_m\}$ and let $l=1,2,...$. Then
$$
\sup\Delta^{k_1k_2...k_m} _{c_1c_2...c_m}-\inf\Delta^{k_1k_2...k_m} _{c_1c_2...c_m}=
$$
$$
=\sum_{2l=j  \notin K, j<k_m} {\frac{d_j-1}{d_1d_2...d_j}}+\frac{(-1)^{k_m}}{d_1d_2...d_{k_m}}\sum^{\infty} _{p=1} {\frac{d_{k_m+2p}-1}{d_{k_m+1}...d_{k_m+2p}}}+
$$
$$
+\frac{(-1)^{k_m}}{d_1d_2...d_{k_m}}\sum^{\infty} _{p=1} {\frac{d_{k_m+2p}-1}{d_{k_m+1}...d_{k_m+2p}}}-
$$
$$
-\sum_{2l+1=j  \notin K, j<k_m} {\frac{1-d_j}{d_1d_2...d_j}}-\frac{(-1)^{k_m}}{d_1d_2...d_{k_m}}\sum^{\infty} _{p=1} {\frac{1-d_{k_m+2p+1}}{d_{k_m+1}...d_{k_m+2p+1}}}=
$$
$$
=\sum_{0<j<k_m, j \notin K} {\frac{d_j-1}{d_1d_2...d_j}}+ \frac{1}{d_1d_2...d_{k_m}}=1-\sum^{m} _{i=1} {\frac{d_{k_i}-1}{d_1d_2...d_{k_i}}}.
$$
\end{proof}

\section{Faithfulness of nega-D-cylinders for\\
 the Hausdorff-Besicovitch dimension calculating}

Let $E$  be a bounded subset of $[a_0-1;a_0]$.

\begin{theorem}
Let a sequence $(d_n)$ of elements of the alternating Cantor series sum is bounded. Then for any $E \subset [a_0-1;a_0]$  values of the Hausdorff-Besicovitch dimension of $E$  calculated by family of cylinders $\Delta^{-D} _{c_1c_2...c_n}$ and calculated by family of closed intervals are equals. 

\end{theorem}

\begin{proof} Let $\Phi_1$ be a family  of covering of  $E$  by closed intervals and $\Phi_2$  be a family  of covering of  $E$  by cylinders $\Delta^{-D} _{c_1c_2...c_n}$.

Let us find conditions on $(d_n)$ such that 
$$
 m^{\alpha} _{\varepsilon} (E,\Phi_1)\le m^{\alpha} _{\varepsilon} (E,\Phi_2).
$$
for $\Phi_2 \subset \Phi_1$.

Let $u$ be an arbitrary closed interval of covering of $E$.

Let $k$ be a minimal positive integer such that $u$  does not contain nega-D-cylinders $\Delta^{-D} _{c_1c_2...c_n}$ of rank $k-1$. Then $u$ belongs to not more than $d_k$ cylinders of rank $k$ but $u$ contains a cylinder of rank $k+1$.
$$
m^{\alpha} _{\varepsilon} (E,\Phi_1)\le m^{\alpha} _{\varepsilon} (E,\Phi_2)\le d_kd_{k+1} m^{\alpha} _{\varepsilon} (E,\Phi_1), 
$$
where
$$
m^{\alpha} _{\varepsilon} (E,\Phi)=\inf_{d(E_j)\le \varepsilon} {\sum_j {d^{\alpha}(E_j)}} 
$$
for fixed  $\varepsilon>0$, fixed $\alpha>0$ and covering of  $E$ by sets  $E_j$ with diameters  $d(E_j)\le \varepsilon$.

It should be noted that 
$$
d_kd_{k+1}\le \left(\max_n \{d_n\}\right)^2<\infty, ~~~\mbox{if  $(d_n)$ is bounded.} 
$$
Indeed,
$$
0<\lambda_1= \frac{1}{\max_n \{d_n\}}\le \frac{|\Delta^{-D} _{c_1c_2...c_ni}|}{|\Delta^{-D} _{c_1c_2...c_n}|}= \frac{1}{d_{n+1}}\le \frac{1}{2}=\lambda_2<1,
$$
where for arbitrary  $n \in  \mathbb N$ $\lambda_1$ and $\lambda_2$ are fixed numbers. It is true iff a sequence $(d_n)$ is  bounded.  
\end{proof}

In 2013, several authors \cite{AILT2013} considered  general necessary and sufficient conditions for a covering family to be faithful and new techniques for proving faithfulness/non-faithfulness for the family of cylinders generated by expansions of real numbers by positive Cantor series.
\section{Set of incomplete sums}

Let $(d_n)$ be a fixed sequence of positive integer numbers $d_n>1$ and $(\varepsilon_n)$ be a fixed sequence. Let us consider
the  corresponding alternating Cantor series 
\begin{equation}
\label{eq:sum1}
s_0=\sum^{\infty} _{n=1} {\frac{(-1)^{n}\varepsilon_n}{d_1d_2...d_n}}\equiv\Delta^{-D} _{\varepsilon_1\varepsilon_2...\varepsilon_n...}.
\end{equation}
Let $(A^{'} _n)$ be a sequence of sets  $A^{'} _{n}\equiv\{0,\varepsilon_n\}$ and
$$
L^{'} _{s_0}=A^{'} _1 \times A^{'} _2 \times A^{'} _n \times ...=\{\delta: \delta=(\delta_1, \delta_2,...,\delta_n,...), \delta_n \in A^{'} _n\}.
$$
\begin{definition}
A number $s=s(\delta)$ 
\begin{equation}
\label{eq:sum2}
s=s(\delta)=\sum^{\infty} _{n=1} {\frac{(-1)^n\delta_n}{d_1d_2...d_n}},
\end{equation}
where $\delta=(\delta_n) \in L^{'} _{s_0}$,
is called {\itshape the  incomplete sum of the alternating Cantor series \eqref{eq:sum1}}. 
\end{definition}

 The set of all incomplete sums of the alternating Cantor series \eqref{eq:sum1} is denoted by $M_{s_0}$, i. e. 
$$
M_{s_0}\equiv\{x: x=\Delta^{-D} _{\delta_1\delta_2...\delta_n...}, (\delta_n)\in L^{'} _{s_0}\}.
$$

It is obvious that 
$$
M_{s_0}\subset \left[-\sum^{\infty} _{k=1} {\frac{\varepsilon_{2k-1}}{d_1d_2...d_{2k-1}}};\sum^{\infty} _{k=1} {\frac{\varepsilon_{2k}}{d_1d_2...d_{2k}}}\right]=I^{M_{s_0}} _0~~~\mbox{for}~~~s_0=\Delta^{-D} _{\varepsilon_1\varepsilon_2...\varepsilon_n...}
$$
and $M_{s_0}=\{0\}$ for $s_0=0$. Moreover,
$$
\bigcup_{s_0} {M_{s_0}}=\left[-\sum^{\infty} _{k=1} {\frac{d_{2k-1}}{d_1d_2...d_{2k-1}}};\sum^{\infty} _{k=1} {\frac{d_{2k}}{d_1d_2...d_{2k}}}\right].
$$

For investigating topological and metric properties of sets $M_{s_0}$  of all incomplete sums of the series \eqref{eq:sum1}, one can  introduce  some auxiliary notions.

\begin{definition}
\label{def:11.7}
Cylinder of rank $n$ with the base $c_1c_2...c_n$ is a set 
$$
\Delta^{M_{s_0}} _{c_1c_2...c_m}\equiv\left\{x: x=\sum^{n} _{i=1} {\frac{(-1)^ic_i}{d_1d_2...d_i}}+\sum^{\infty} _{j=n+1} {\frac{(-1)^j\delta_j}{d_1d_2...d_j}}\right\},
$$
where $c_1,c_2,...,c_n$ are fixed numbers from $A^{'} _1,A^{'} _2,...,A^{'} _n$ respectively and $\delta_j \in A^{'} _j$.
\end{definition}

\begin{definition}
Cylindrical closed interval (interval)  $I^{M_{s_0}} _{c_1c_2...c_n}$ ($\nabla^{M_{s_0}} _{c_1c_2...c_n}$) of rank  $n$ with the base  $c_1c_2...c_n$ is a closed interval (interval), that its ends coincide with ends of the cylinder $\Delta^{M_{s_0}} _{c_1c_2...c_n}$.
\end{definition}

The following properties of cylindrical sets follow from the Definition~\ref{def:11.7} immediately.
\begin{enumerate}
\item
$$
\inf {\Delta^{M_{s_0}} _{c_1c_2...c_n}}=\begin{cases}
\Delta^{-D} _{c_1c_2...c_n 0\varepsilon_{n+2}0\varepsilon_{n+4}...},&\text{if $n$ be an odd;}\\
\Delta^{-D} _{c_1c_2...c_n\varepsilon_{n+1}0\varepsilon_{n+3}0\varepsilon_{n+5}...},&\text{if $n$ be an even.}
\end{cases}
$$
\item
$$
\sup {\Delta^{M_{s_0}} _{c_1c_2...c_n}}=\begin{cases}
\Delta^{-D} _{c_1c_2...c_n\varepsilon_{n+1}0\varepsilon_{n+3}0\varepsilon_{n+5}...},&\text{if $n$ be an odd;}\\
\Delta^{-D} _{c_1c_2...c_n 0\varepsilon_{n+2}0\varepsilon_{n+4}...},&\text{if $n$ be an even.}
\end{cases}
$$
\item
$$
d(\Delta^{M_{s_0}} _{c_1c_2...c_n})=\Delta^{D} _{\underbrace{0...0}_{n}\varepsilon_{n+1}\varepsilon_{n+2}...}\equiv \sum^{\infty} _{k=n+1} {\frac{\varepsilon_{k}}{d_1d_2...d_k}}\le \frac{1}{d_1d_2...d_n}\to 0, 
$$
$n\to \infty$.
\item If $\varepsilon_{n+1}\ne 0$,
$$
\Delta^{M_{s_0}} _{c_1c_2...c_n}=\Delta^{M_{s_0}} _{c_1c_2...c_n0}\cup \Delta^{M_{s_0}} _{c_1c_2...c_n\varepsilon_{n+1}}.
$$
\item \label{pro5}
$$
\Delta^{M_{s_0}} _{c_1c_2...c_n}\subset I^{M_{s_0}} _{c_1c_2...c_n}\subset \Delta^{-D} _{c_1c_2...c_n},
$$
$$
M_{s_0}\subset \bigcup _{c_i \in A^{'} _i, i=\overline{1,n}} {\Delta^{M_{s_0}} _{c_1c_2...c_n}}\subset  \bigcup_{c_i \in A^{'} _i, i=\overline{1,n}} {I^{M_{s_0}} _{c_1c_2...c_n}}.
$$
\item 
$$
|\Delta^{-D} _{c_1c_2...c_n}\setminus (\Delta^{-D} _{c_1c_2...c_n0} \cup \Delta^{-D} _{c_1c_2...c_n\varepsilon_{n+1}})|=\begin{cases}
\frac{d_{n+1}-2}{d_1d_2...d_{n+1}},&\text{if $\varepsilon_{n+1}>0$;}\\
\frac{d_{n+1}-1}{d_1d_2...d_{n+1}},&\text{if $\varepsilon_{n+1}=0$.}
\end{cases}
$$
 \end{enumerate}

\begin{lemma} Let $s_0=\Delta^{-D} _{\varepsilon_1\varepsilon_2...\varepsilon_n...}$ be a given number and let $(c_n)$ be an arbitrary fixed sequence from  $L^{'} _{s_0}$. Then
\begin{enumerate}
\item
$$
\bigcap^{\infty} _{i=1} {\Delta^{-D} _{c_1c_2...c_n}}=\bigcap^{\infty} _{i=1} {\Delta^{M_{s_0}} _{c_1c_2...c_n}}=\Delta^{-D} _{c_1c_2...c_n...}.
$$
\item
$$
\Delta^{M_{s_0}} _{c_1c_2...c_n}=\Delta^{-D} _{c_1c_2...c_n}\cap M_{s_0}
$$
and
$$
M_{s_0}=\bigcap^{\infty} _{n=1} {\left(\bigcup_{c_i \in A^{'} _{i}, i=\overline{1,n}} {\Delta^{M_{s_0}} _{c_1c_2...c_n}}\right)},
$$
where $A^{} _i=\{0,\varepsilon_i\}$.
\end{enumerate}
\end{lemma}
\begin{proof}

1. Let $x=\Delta^{-D} _{c_1c_2...c_n...}$. $\Delta^{-D} _{c_1c_2...c_n...}=x \in \Delta^{M_{s_0}} _{c_1c_2...c_n}$ follows from the definition of $\Delta^{M_{s_0}} _{c_1c_2...c_n}$. Therefore, $x \in \bigcap^{\infty} _{n=1} {I^{M_{s_0}} _{c_1c_2...c_n}}$. First proposition of the lemma follows from the Property \ref{pro5} of $\Delta^{M_{s_0}} _{c_1c_2...c_n}$.

2.  Let $x \in M_{s_0}$, then  $x$ belongs to certain cylinder  $\Delta^{M_{s_0}} _{c_1c_2...c_n}\subset \Delta^{-D} _{c_1c_2...c_n}$. Let us consider a set $\Delta^{-D} _{c_1c_2...c_n} \cap M_{s_0}$. Numbers of type  $\Delta^{-D} _{c_1c_2...c_n\delta_{n+1}\delta_{n+2}...\delta_{n+k}...}$, where $\delta_{n+k} \in A^{'} _{n+k}$, are elements of the set. 
Consequently,
$$
\Delta^{M_{s_0}} _{c_1c_2...c_n}\subset (\Delta^{-D} _{c_1c_2...c_n} \cap M_{s_0})
$$
 and  if $x \in (\Delta^{-D} _{c_1c_2...c_n} \cap M_{s_0})$, then $x \in \Delta^{M_{s_0}} _{c_1c_2...c_n}$.
\end{proof}

\begin{theorem}
The set  $ M_{s_0}$ of incomplete sums of the alternating Cantor series is 
\begin{enumerate}
\item an one-element set $\{0\}$, if $s_0=0$;

\item a finite set, when the condition $\varepsilon_n \ne 0$ is true for finite set of values $n$ in $s_0=\Delta^{-D} _{\varepsilon_1\varepsilon_2...\varepsilon_n...}$;
 
\item
a segment $\left[-\frac{2}{3};\frac{1}{3}\right]$, when   $d_n=const=2$ for all \  $n \in \mathbb N$ \  and \  $s_0=-\frac{1}{3}$;

\item a union of finite number of  segments, when there exists finite set of values $m_i$ ($i=\overline{1,k_0}$, $k_0$ is a fixed number), that $d_{m_i}\ne 2$ and $s_0=\Delta^{-D} _{\varepsilon_1\varepsilon_2...\varepsilon_{m_{k_0}}(1)}$;

\item a continuous, perfect, nowhere dense set of zero Lebesgue measure, when $s_0 \ne 0$ and $d_n>2$ for infinite set of values  $n$. 
\end{enumerate}
\end{theorem}
\begin{proof} First four propositions follow from the following facts: the set $M_{s_0}$ is $C[-D, A^{'} _n]$, for condition  $d_n=2$ ( $n \in \mathbb N$) the alternating Cantor series is a nega-binary sum
$$
-\frac{\varepsilon_1}{2}+\frac{\varepsilon_2}{2^2}-\frac{\varepsilon_3}{2^3}+...+\frac{(-1)^n \varepsilon_n}{2^n}+...,~\mbox{де}~ \varepsilon_n \in \{0,1\}.
$$

To prove fifth proposition, let us consider
$$
x=\Delta^{-D} _{\delta_1\delta_2...\delta_n...} \stackrel{f}{\rightarrow} \sum^{\infty} _{n=1} {\frac{1}{(2+\delta_1)(2+\delta_1+\delta_2)...(2+\delta_1+...+\delta_n)}}\equiv
$$
$$
\equiv \Delta^E _{\delta_1\delta_2...\delta_n...}=f(x)=y.
$$
That is argument of  the mapping $f$ is represented by the alternating Cantor series and its value is represented by the positive Engel series and $ f: M_{s_0} \to C[E,V_n]$, where $V_n=~A^{'} _n$. The last-mentioned mapping is not bijection in  nega-D-rational points
$$
\Delta^{-D} _{\delta_1\delta_2...\delta_{k-1}\delta_k[d_{k+1}-1]0[d_{k+3}-1]0...}=\Delta^{-D} _{\delta_1\delta_2...\delta_{k-1}[\delta_k-1]0[d_{k+2}-1]0[d_{k+4}-1]...}.
$$
It is true, when $s_0=\Delta^{-D} _{\varepsilon_1\varepsilon_2...\varepsilon_{k-1}1[d_{k+1}-1][d_{k+2}-1][d_{k+3}-1]...}$. 
Continuum power  of $M_{s_0}$ follows from continuity of $ C[E,V_n]$ because a set of  nega-D-rational numbers is no-more-than countable set in $M_{s_0}$ and one can to use one from two representations of the nega-D-rational number (for example, first) for a case, when an argument   is a such number.         

Let us  prove that the set  $M_{s_0}$ is a  nowhere dense set.  

Let us choose a cylinder  $\Delta^{M_{s_0}} _{c_1c_2...c_{n-1}}$, that a condition $\varepsilon_n \ne 0$ is true for $s_0=\Delta^{-D} _{\varepsilon_1\varepsilon_2...\varepsilon_n...}$. Let us  investigate the mutual placement of cylinders 
 $\Delta^{M_{s_0}} _{c_1c_2...c_{n-1}0}$, $\Delta^{M_{s_0}} _{c_1c_2...c_{n-1}\varepsilon_{n}}$. Let  $n$ be an even number, then 
$$
\inf \Delta^{M_{s_0}} _{c_1c_2...c_{n-1}\varepsilon_{n}}-\sup \Delta^{M_{s_0}} _{c_1c_2...c_{n-1}0}=\sum^{n-1} _{i=1} {\frac{(-1)^ic_i}{d_1d_2...d_i}}+\frac{\varepsilon_{n}}{d_1d_2...d_{n}}-
$$
$$
-\sum^{\infty} _{k=1} {\frac{\varepsilon_{n+2k-1}}{d_1d_2...d_{n+2k-1}}}-\sum^{n-1} _{i=1} {\frac{(-1)^ic_i}{d_1d_2...d_i}}-\sum^{\infty} _{k=1} {\frac{\varepsilon_{n+2k}}{d_1d_2...d_{n+2k}}}=
$$
$$
=\frac{\varepsilon_{n}}{d_1d_2...d_{n}}-\sum^{\infty} _{k=1} {\frac{\varepsilon_{n+k}}{d_1d_2...d_{n+k}}}=\frac{1}{d_1d_2...d_n}\left(\varepsilon_n-\sum^{\infty} _{k=1} {\frac{\varepsilon_{n+k}}{d_{n+1}...d_{n+k}}}\right)\ge 0.
$$
That is  cylinders are left-to-right situated  and the last-mentioned difference  equals to zero, when $s_0=\Delta^{-D} _{\varepsilon_1\varepsilon_2...\varepsilon_{n-1}1[d_{n+1}-1][d_{n+2}-1][d_{n+3}-1]...}$. 

Analogously, for the  inequality   
$$
\inf \Delta^{M_{s_0}} _{c_1c_2...c_{n-1}0}-\sup \Delta^{M_{s_0}} _{c_1c_2...c_{n-1}\varepsilon_{n}}\ge 0,
$$
the same result is obtained in the case of odd $n$.  That is  cylinders $\Delta^{M_{s_0}} _{c_1c_2...c_{n}}$ are right-to- left  situated. 
So, for any interval  belonging to  $[\inf M_{s_0};\sup M_{s_0}]$ there exists subinterval, that the subinterval  does not contain points from $M_{s_0}$ because  $\Delta^{M_{s_0}} _{c_1c_2...c_{n-1}0}\cap~\Delta^{M_{s_0}} _{c_1c_2...c_{n-1}\varepsilon_{n}}\ne~\varnothing $ iff 
$$
\varepsilon_n=0  ~\mbox{or}~  s_0=\Delta^{-D} _{\varepsilon_1\varepsilon_2...\varepsilon_{n-1}1[d_{n+1}-1][d_{n+2}-1][d_{n+3}-1]...}.
$$

Let us prove that  $M_{s_0}$ is a closed  set without isolated points. Let us choose an arbitrary limit   point $x_0$ from $M_{s_0}$.
From definition of the point it follows that interval $(x_0-\varepsilon;x_0+~\varepsilon)$ contains at least one point from $M_{s_0}$ (not coinciding with $x_0$) for all  $\varepsilon >0$. If there   does not exist  unique closed interval $I^{M_{s_0}} _{\delta_1(x_0)\delta_2(x_0)...\delta_n(x_0)}$, that $x_0$ belongs to the closed interval, then $x_0$ belongs to one from adjacent to $M_{s_0}$ intervals. Therefore, there exists $\varepsilon_0 >0$ such that $(x_0-~\varepsilon_0;x_0+~\varepsilon_0)\cap~M_{s_0}=~\varnothing$. In this case $x_0$ is not limit. If there exists 
closed interval $I^{M_{s_0}} _{\delta_1(x_0)\delta_2(x_0)...\delta_n(x_0)}$, then 
$$
M_{s_0} \ni \bigcap^{\infty} _{n=1} {I^{M_{s_0}} _{\delta_1(x_0)\delta_2(x_0)...\delta_n(x_0)}}=x_0.
$$

Hence, $M_{s_0}$ is a closed set. 

Suppose that there exists certain  isolated point  $x^{'}=\Delta^{-D} _{\delta_1\delta_2...\delta_n...}$. Then there exists   $\varepsilon_0 >0$ such that 
\begin{equation}
\label{8}
(x^{'}-\varepsilon_0;x^{'}+\varepsilon_0)\cap (M_{s_0}\setminus \{x^{'}\})=\varnothing. 
\end{equation}
Let us choose a number $m$ such that \   $d(\Delta^{M_{s_0}} _{\delta_1\delta_2...\delta_m})<\varepsilon_0$ \ and  \  $\varepsilon_{m+1}(s_0)\ne 0$. Then $\Delta^{M_{s_0}} _{\delta_1\delta_2...\delta_m}\subset~(x^{'}-\varepsilon_0;x^{'}+\varepsilon_0)$ and
$$
x^{'}\ne x =\Delta^{-D} _{\delta_1\delta_2...\delta_n\sigma\delta_{m+2}...}\in (x^{'}-\varepsilon_0;x^{'}+\varepsilon_0)\cap M_{s_0},~\mbox{where}
$$
$$
\sigma=\begin{cases}
\varepsilon_{m+1},&\text{if $\delta_{m+1}=0$;}\\
0,&\text{if $\delta_{m+1}\ne 0$.}
\end{cases}
$$
The last contradicts \eqref{8}. So, the assumption  is  false. The set $M_{s_0}$ does  not contain  isolated points. 

Let us calculate a Lebesgue measure of  $M_{s_0}$.  Let  $F_k$  be a  union of cylinder closed intervals $I^{M_{s_0}} _{c_1c_2...c_k}$ of rank $k$ ($c_k \in A^{'} _k$). Then $M_{s_0}\subset F_k \subset F_{k+1}$ for all $k \in \mathbb N$ and $\lambda (M_{s_0})\le~\lim_{k \to \infty} {F_k}$. 
From $d(\Delta^{M_{s_0}} _{c_1c_2...c_n})=~|I^{M_{s_0}} _{c_1c_2...c_n}|$ and properties of $\Delta^{M_{s_0}} _{c_1c_2...c_n}$  it follows that
$$
\lambda (M_{s_0})\le \lim_{k \to \infty} {\left(2^k \cdot \Delta^D _{\underbrace{0...0}_{k}\varepsilon_{k+1}(s_0)\varepsilon_{k+2}(s_0)\varepsilon_{k+3}(s_0)...}\right)}=
$$
$$
=\lim_{k \to \infty} {\left(\frac{2^k}{d_1d_2...d_k}\cdot \sum^{\infty} _{i=k+1} {\frac{\varepsilon_{i}(s_0)}{d_{k+1}...d_i}}\right)}=0.
$$
\end{proof}

\end{document}